\documentclass[3p,times]{elsarticle}

\usepackage{ecrc}

\volume{253}

\firstpage{1}

\journalname{Journal of Number Theory}

\runauth{C. D. Zhang, L. Yang}

\jid{jnt}

\jnltitlelogo{Journal of Number Theory}

\CopyrightLine{2023}{Published by Elsevier Ltd.}

\usepackage{amsmath,amssymb}
\usepackage{amsthm}

\theoremstyle{plain}
\newtheorem{theorem}{Theorem}[section]
\theoremstyle{definition}

\usepackage[figuresright]{rotating}

\begin{document}

\begin{frontmatter}
	
	
	
	\dochead{}
	
	\title{Some New Ramanujan’s Modular Equations of Degree 15}
	\tnotetext[NSFC]{This work was supported by the National Natural Science Foundation of China (Grant No. 42174037).}	
	
	\author[1,2]{Zhang Chuan-Ding \corref{cor1}}
	\cortext[cor1]{Corresponding author}
	\ead{13607665382@163.net}

	\author[1,2]{Yang Li}
	\ead{241912296@qq.com}
	
	\address[1]{College of Geography and Environmental Science, Henan University, Zhengzhou 450046, China}
	\address[2]{Henan Technology Innovation Center of Spatio-Temporal Big Data, Henan University, Zhengzhou 450046, China}

	\begin{abstract}
		Ramanujan in his notebook recorded two modular equations involving multiplier with moduli of degrees (1,7)  and (1,23).   
		In this paper, we find some new Ramanujan’s modular equations  involving multiplier with moduli of degrees (3,5)  and (1,15), and give concise proofs by employing Ramanujan’s multiplier function equation.
	\end{abstract}
	
	\begin{keyword}
		Theta functions \sep Modular equations \sep Multiplier
		
		
		\MSC[2000] 11F20 \sep 33C05 
		
	\end{keyword}
	
\end{frontmatter}



\section{Introduction}
\label{}

Ramanujan defined modular equations as follows: Suppose that
\begin{equation}\label{eq01}
	n \frac{{}_2 F_1\left(\frac{1}{2}, \frac{1}{2} ; 1 ; 1-\alpha\right)}{{}_2 F_1\left(\frac{1}{2}, \frac{1}{2} ; 1 ; \alpha\right)}=\frac{{}_2 F_1\left(\frac{1}{2}, \frac{1}{2} ; 1 ; 1-\beta\right)}{{}_2 F_1\left(\frac{1}{2}, \frac{1}{2} ; 1 ; \beta\right)},
\end{equation}
holds for some positive integer $n$ or positive rational fraction $n=i/j$ (i and j are coprime). The relation between $\alpha$ and $\beta$ induced by the above equation is called a modular equation of degree $n$ and in such equations, we say that $\beta$ over $\alpha$ is degree $n$. Let
\begin{equation}\label{eq02}
	z_1={}_2 F_1\left(\frac{1}{2}, \frac{1}{2} ; 1 ; \alpha\right) ,\quad z_n={}_2 F_1\left(\frac{1}{2}, \frac{1}{2} ; 1 ; \beta\right),
\end{equation}
the multiplier $m$ of the modular equation is defined by
\begin{equation}\label{eq03}
	m:=\frac{z_1}{z_n}.
\end{equation}

Before proceeding to the main theta function identity of this paper, we shall first
recall certain known theta function identities which we need in the sequel. Throughout
the paper, we assume $|q|<1$. The following is the well-known Jacobi triple product
identity:
\begin{equation}\label{eq04}
	f(a, b):=\sum_{n=-\infty}^{\infty} a^{n(n+1) / 2} b^{n(n-1) / 2}=(-a ; a b)_{\infty}(-b ; a b)_{\infty}(a b ; a b)_{\infty}, \quad|a b|<1,
\end{equation}
the two particular facts of $f(a, b)$ (See \cite{Bruce1991},Entry 22, p36), are as follows
\begin{equation}\label{eq05}
	\varphi(q):=f(q, q)=\left(-q ; q^2\right)_{\infty}^2\left(q^2 ; q^2\right)_{\infty}, \quad \psi(q):=f\left(q, q^3\right)=\left(-q ; q^2\right)_{\infty}\left(q^2 ; q^2\right)_{\infty}
	.
\end{equation}

Again, let the base $q$ in the classical theory of elliptic functions is defined by 
\begin{equation}\label{eq06}
	q=\exp \left(-\pi\frac{{}_2 F_1\left(\frac{1}{2}, \frac{1}{2} ; 1 ; 1-\alpha\right)}{{}_2 F_1\left(\frac{1}{2}, \frac{1}{2} ; 1 ; \alpha\right)}\right),
\end{equation}
we have the following well-known identities (See \cite{Bruce1991},Entry 10 and 11, pp. 122-123)
\begin{equation}\label{eq07}
	z_1=\varphi^2(q),\quad z_n=\varphi^2(q^n),\quad m=\frac{\varphi^2(q)}{\varphi^2(q^n)},
\end{equation}
and
\begin{equation}\label{eq08}
	\frac{\sqrt[4]{q} \psi(q^2)}{\varphi (q)}= \frac{\sqrt[4]{\alpha }}{2},\quad \frac{\varphi (-q)}{\varphi (q)}= \sqrt[4]{1-\alpha },\quad \frac{q^{n/4} \psi (q^{2 n})}{\varphi \left(q^n\right)}= \frac{\sqrt[4]{\beta }}{2},\quad \frac{\varphi \left(-q^n\right)}{\varphi \left(q^n\right)}= \sqrt[4]{1-\beta },
\end{equation}
and
\begin{equation}\label{eq09}
	\frac{\sqrt[8]{q} \psi (q)}{\varphi (q)}= \frac{\sqrt[8]{\alpha }}{\sqrt{2}},\quad \frac{\varphi (-q^2)}{\varphi (q)}= \sqrt[8]{1-\alpha },\quad \frac{q^{n/8} \psi(q^n)}{\varphi \left(q^n\right)}= \frac{\sqrt[8]{\beta }}{\sqrt{2}},\quad \frac{\varphi (-q^{2 n})}{\varphi \left(q^n\right)}= \sqrt[8]{1-\beta }.
\end{equation}

Ramanujan found and recorded the following  modular equations involving multiplier with moduli of degrees 7  and 23,in tha case  degree $(n \bmod 8)=7$, respectively.
\begin{theorem}[See \cite{Bruce1991},Entry 19, pp. 314-315]If $\beta$ has degree 7 over $\alpha$, and $m$ is the multiplier for degree 7 , then
	\begin{equation}\label{eq10}
		(\alpha \beta)^{1 / 8}+\{(1-\alpha)(1-\beta)\}^{1 / 8}=1,
	\end{equation}
	and
	\begin{equation}\label{eq11}
		m-\frac{7}{m}=2\left((\alpha \beta)^{1 / 8}-\{(1-\alpha)(1-\beta)\}^{1 / 8}\right)\left(2+(\alpha \beta)^{1 / 4}+\{(1-\alpha)(1-\beta)\}^{1 / 4}\right),
	\end{equation}
	and
	\begin{equation}\label{eq12}
		\frac{\varphi^2(q)}{\varphi^2(q^7)}-7 \frac{\varphi^2(q^7)}{\varphi^2(q)}=2\left(2 q \frac{\psi(q) \psi(q^7)}{\varphi(q)\varphi(q^7)} -\frac{\varphi (-q^2)\varphi (-q^{14})}{\varphi(q)\varphi(q^7)}\right) \left(2+4 q^2 \frac{\psi(q^2) \psi(q^{14})}{\varphi(q)\varphi(q^7)}+\frac{\varphi (-q)\varphi (-q^7)}{\varphi(q)\varphi(q^7)}\right).
	\end{equation}
\end{theorem}
\begin{theorem}[See \cite{Bruce1991},Entry 15, p.411]
	If $\beta$ is of the 23 degree over $\alpha$,  and $m$ is the multiplier for same degree, then
	
	\begin{equation}\label{eq13}
		(\alpha \beta)^{1 / 8}+\{(1-\alpha)(1-\beta)\}^{1 / 8}+2^{2 / 3}\{\alpha \beta(1-\alpha)(1-\beta)\}^{1 / 24}=1,
	\end{equation}
	and
	\begin{equation}\label{eq14}
		\begin{aligned}
			m-\frac{23}{m} &= 2\left((\alpha \beta)^{1 / 8}-\{(1-\alpha)(1-\beta)\}^{1 / 8}\right)\left(11-13 \cdot 4^{1 / 3}\{\alpha \beta(1-\alpha)(1-\beta)\}^{1 / 24}\right. \\
			& +18 \cdot 2^{1 / 3}\{\alpha \beta(1-\alpha)(1-\beta)\}^{1 / 12}-14\{\alpha \beta(1-\alpha)(1-\beta)\}^{1 / 8} \left.+2^{5 / 3}\{\alpha \beta(1-\alpha)(1-\beta)\}^{1 / 6}\right) .
		\end{aligned}
	\end{equation}
\end{theorem}
Berndt finished the proof of (\ref{eq14}) by the theory of modular forms (See \cite{Bruce1991},Entry 15, pp. 415-416). In this paper, we also find some new Ramanujan’s modular equations involving multiplier with moduli of degrees (3,5) and (1,15).

\section{Some New Ramanujan’s Modular Equations of  degrees 15 and $\frac{5}{3}$}
The beautiful Ramanujan’s modular equations of degrees 15 and $\frac{5}{3}$ (See \cite{Bruce1991},Entry 21, p435): Let $\alpha$ and $\beta$ have degrees $(n_1,n_2)=(3,5)$; or $(1, 15)$, respectively. Then
\begin{equation}\label{eq15}
	(\alpha \beta)^{1 / 8}+\{(1-\alpha)(1-\beta)\}^{1 / 8} \pm\{\alpha \beta(1-\alpha)(1-\beta)\}^{1 / 8} =\sqrt{\frac{1}{2}(1+\sqrt{\alpha \beta}+\sqrt{(1-\alpha)(1-\beta)})},
\end{equation}
where the minus sign is chosen in the first case and the plus sign is selected in the last case.

For modular equations in the form of Russell, we redefine (See \cite{Bruce1998},Entry 21, p435),
\begin{equation}\label{eq16}
	\left\{
	\begin{aligned}
		& P=1+(-1)^{\frac{n1+n2}{8}} \left((\alpha \beta)^{1 / 8}+\{(1-\alpha)(1-\beta)\}^{1 / 8}\right),
		\\
		& Q=4\left((\alpha \beta)^{1 / 8}+\{(1-\alpha)(1-\beta)\}^{1 / 8} +(-1)^{\frac{n1+n2}{8}} \{\alpha \beta(1-\alpha)(1-\beta)\}^{1 / 8}\right),\\
		&R=4(\alpha \beta(1-\alpha)(1-\beta))^{1 / 8} ,
	\end{aligned}\right.
\end{equation}
then, if $\alpha$ and $\beta$ have degrees $(n_1,n_2)=(1, 15)$ (See \cite{Bruce1998},Entry 21, p435),
\begin{equation}\label{eq17}
	P(P^2-Q)+R=0,
\end{equation}
and, if $\alpha$ and $\beta$ have degrees $(n_1,n_2)=(3, 5)$,
\begin{equation}\label{eq18}
	P(P^2+Q)+R=0.
\end{equation}
We can verify that these two modular equations (\ref{eq17}) and (\ref{eq18}) are equivalent formulations of (\ref{eq15}).

\begin{theorem}[new modular equations for degrees $(n_1,n_2)=(1, 15)$] Let $\alpha$ and $\beta$ have degrees $(n_1,n_2)=(1, 15)$, and $m=\frac{z_1}{z_{15}}$ is the multiplier for degree $n=15$, then 
	\begin{enumerate}[(i)]
		\item natural form	
		\begin{equation}\label{eq19}	
			\begin{aligned}
				m-\frac{15}{m} =&2 \left((\alpha \beta)^{1 / 8}-\{(1-\alpha)(1-\beta)\}^{1 / 8}\right) \left[1+3\left((\alpha \beta)^{1 / 8}+\{(1-\alpha)(1-\beta)\}^{1 / 8}\right)\right.\\
				&\left.+3\left((\alpha \beta)^{1 / 4}+\{(1-\alpha)(1-\beta)\}^{1 / 4}\right)+2 (\alpha \beta (1-\alpha)(1-\beta))^{1 / 8}\left(3+(\alpha \beta)^{1 / 8}+\{(1-\alpha)(1-\beta)\}^{1 / 8}\right) \right],
			\end{aligned}
		\end{equation}		
		\item emphatic form		
		\begin{equation}\label{eq20}	
			\begin{aligned}
				m-\frac{15}{m} =&2 \left((\alpha \beta)^{1 / 4}-\{(1-\alpha)(1-\beta)\}^{1 / 4}\right) \left[4\sqrt{\frac{1}{2}(1+\sqrt{\alpha \beta}+\sqrt{(1-\alpha)(1-\beta)})}\right.\\
				&\left.+4-\left((\alpha \beta)^{1 / 4}+\{(1-\alpha)(1-\beta)\}^{1 / 4}\right) \right],
			\end{aligned}
		\end{equation}	
		\item Ramanujan’s theta functions form	
		\begin{equation}\label{eq21}	
			\begin{aligned}
				\frac{\varphi^2(q)}{\varphi^2(q^{15})}-15 \frac{\varphi^2(q^{15})}{\varphi^2(q)}=&2 \left(4 q^4 \frac{\psi(q^2) \psi(q^{30})}{\varphi(q)\varphi(q^{15})}-\frac{\varphi (-q)\varphi (-q^{15})}{\varphi(q)\varphi(q^{15})}\right) \left[4\left(\frac{\varphi (q^2) \varphi (q^{30})}{\varphi (q) \varphi \left(q^{15}\right)}+4 q^8\frac{ \psi (q^4) \psi(q^{60})}{\varphi (q) \varphi \left(q^{15}\right)}\right)\right.\\
				&\left.+4-\left(4 q^4 \frac{\psi(q^2) \psi(q^{30})}{\varphi(q)\varphi(q^{15})}+\frac{\varphi (-q)\varphi (-q^{15})}{\varphi(q)\varphi(q^{15})}\right) \right].
			\end{aligned}
		\end{equation}	
	\end{enumerate}
	
\end{theorem}

\begin{proof} (i).	According to Ramanujan’s multiplier function equation (See \cite{Bruce1991},Entry 24(vi), p. 217), we can write  
	\begin{equation}\label{eq22}
		n \frac{\operatorname{d} \alpha}{\operatorname{d} \beta}=\frac{\alpha(1-\alpha)}{\beta(1-\beta)} m^2 .
	\end{equation}
	
	Let 
	\begin{equation}\label{eq23}
		\begin{aligned}
			&x:=x(t)=(\alpha \beta)^{1 / 8},\\
			&y:= y(t) =\{(1-\alpha)(1-\beta)\}^{1 / 8},
		\end{aligned}
	\end{equation}
	we can get 
	\begin{equation} \label{eq24}
		\begin{aligned}
			&\alpha:=\alpha(t) = \frac{1}{2} \left(1+x^8-y^8+\sqrt{1-2 x^8-2 y^8-2 x^8 y^8+x^{16}+y^{16}}\right),\\
			&\beta:=\beta(t) = \frac{1}{2} \left(1+x^8-y^8-\sqrt{1-2 x^8-2 y^8-2 x^8 y^8+x^{16}+y^{16}}\right),
		\end{aligned}
	\end{equation}
	and using the equations (\ref{eq22}) - (\ref{eq24}), we deduce
	\begin{equation}\label{eq25}
		\frac{n}{m^2} = \frac{\alpha(1-\alpha)  }{\beta(1-\beta)} \frac{\operatorname{d}\beta/\operatorname{d} t}{\operatorname{d}\alpha/\operatorname{d} t}=-\frac{\alpha  y x'(t)+(1-\alpha ) x y'(t)}{\beta  y x'(t)+(1-\beta ) x y'(t)}, 
	\end{equation}
	and	
	\begin{equation}\label{eq26}
		\left(m-\frac{n}{m}\right)^2=-\frac{n \left(y (\alpha +\beta ) x'(t)+x (-\alpha -\beta +2) y'(t)\right)^2}{\left(\alpha  y x'(t)+(1-\alpha ) x y'(t)\right) \left(\beta  y x'(t)+(1-\beta ) x y'(t)\right)}.
	\end{equation}
	
	Again, from the equation (\ref{eq17}), we have equivalent modular equation of degrees 15 in the form of $x$ and $y$, it reads
	\begin{equation}\label{eq27}
		1 - x - x^2 + x^3 - y - 2 x y - x^2 y - y^2 - x y^2 + y^3=0,
	\end{equation}
	from the equation (\ref{eq19}) and (\ref{eq23}), we also have
	\begin{equation}\label{eq28}
		m-\frac{15}{m} =2(x-y) (1 + 3 x + 3 x^2 + 3 y + 6 x y + 2 x^2 y + 3 y^2 + 2 x y^2).
	\end{equation}
	
	Now set
	\begin{equation}\label{eq29}
		x = \frac{1}{2} \left(\frac{1}{t}-\varrho \right),\quad y = \frac{1}{2} \left(\frac{1}{t}+\varrho \right),
	\end{equation}
	using the equation (\ref{eq27}), we obtain
	\begin{equation}\label{eq30}
		t \varrho^2= 1 + t - t^2,
	\end{equation}
	solving the above equation for $\varrho$ and notice that $x<y$, we get
	\begin{equation}\label{eq31}
		\varrho = \frac{\sqrt{1+t-t^2}}{\sqrt{t}},x = \frac{1}{2} \left(\frac{1}{t}-\frac{\sqrt{1+t-t^2}}{\sqrt{t}}\right),\quad y = \frac{1}{2} \left(\frac{1}{t}+\frac{\sqrt{1+t-t^2}}{\sqrt{t}} \right).
	\end{equation}
	
	Employing (\ref{eq31}) and (\ref{eq24}) in (\ref{eq26}) with $n=15$, we deduce
	\begin{equation}\label{eq32}
		\left(m-\frac{15}{m}\right)^2= \frac{\left(1+t-t^2\right) \left(1 + 5 t + 5 t^2 + 3 t^3\right)^2}{t^7},
	\end{equation}
	and employing (\ref{eq31}) in (\ref{eq28}), we obtain 
	\begin{equation}\label{eq33}
		m-\frac{15}{m} =-\frac{\sqrt{1+t-t^2} \left(1 + 5 t + 5 t^2 + 3 t^3\right)}{t^{7/2}}.
	\end{equation}

We can verify that these two equations (\ref{eq32}) and (\ref{eq33}) are equivalent.	So far, we employed the method of parameterization to prove the modular equation of degree 15 which involves the multiplier $m$.
\end{proof}

\begin{proof} (ii).
	Substituting (\ref{eq27}) into (\ref{eq28}),  we obtain
	\begin{equation}\label{eq34}
		\begin{aligned}
			L&=(1 + 3 x + 3 x^2 + 3 y + 6 x y + 2 x^2 y + 3 y^2 + 2 x y^2)-0\\
			&=L-(1 - x - x^2 + x^3 - y - 2 x y - x^2 y - y^2 - x y^2 + y^3)\\
			&=(x + y) (4 (x + y + x y)+4 - x^2 - y^2),
		\end{aligned}
	\end{equation}
	By using the equation (\ref{eq15}), we can write 
	\begin{equation}\label{eq35}
		x+y+xy =\sqrt{\frac{1}{2}(1+\sqrt{\alpha \beta}+\sqrt{(1-\alpha)(1-\beta)})},
	\end{equation}
	Substituting (\ref{eq34}) and (\ref{eq35}) into  (\ref{eq28}),  we arrive at
	\begin{equation}\label{eq36}
		m-\frac{15}{m} =2(x^2-y^2) (4 \sqrt{\frac{1}{2}(1+\sqrt{\alpha \beta}+\sqrt{(1-\alpha)(1-\beta)})}+4 - x^2 - y^2).
	\end{equation}
	which completes the proof.	
\end{proof}	

\begin{proof} (iii).If $\beta$ over $\alpha$ is degree 15, we employ the identity (See \cite{Bruce1991},p. 433 eq. (20.6))
	\begin{equation}\label{eq37}	
		\sqrt{\frac{1}{2}(1+\sqrt{\alpha \beta}+\sqrt{(1-\alpha)(1-\beta)})}=\frac{\sqrt{1-\sqrt{1-\alpha }}}{\sqrt{2}}\frac{\sqrt{1-\sqrt{1-\beta }}}{\sqrt{2}}+\frac{\sqrt{1+\sqrt{1-\alpha }}}{\sqrt{2}}\frac{\sqrt{1+\sqrt{1-\beta }}}{\sqrt{2}},
	\end{equation}
	and translate this identity in form of Ramanujan’s theta functions (See \cite{Bruce1991},Entry 10 and 11, pp. 122-123)
	\begin{equation}\label{eq38}	
		\sqrt{\frac{1}{2}(1+\sqrt{\alpha \beta}+\sqrt{(1-\alpha)(1-\beta)})}=\frac{\varphi (q^2) \varphi (q^{30})}{\varphi (q) \varphi \left(q^{15}\right)}+4 q^8\frac{ \psi (q^4) \psi(q^{60})}{\varphi (q) \varphi \left(q^{15}\right)},
	\end{equation}
	and  use (\ref{eq08}), we get
	\begin{equation}\label{eq39}	
		(\alpha \beta)^{1 / 4}=4 q^4 \frac{\psi(q^2) \psi(q^{30})}{\varphi(q)\varphi(q^{15})},\quad\{(1-\alpha)(1-\beta)\}^{1 / 4}=\frac{\varphi (-q)\varphi (-q^{15})}{\varphi(q)\varphi(q^{15})}.
	\end{equation}
	Substituting (\ref{eq38}) and (\ref{eq39}) into (\ref{eq20}),	which completes the proof.	
\end{proof}

\begin{theorem}[new modular equations for degrees $(n_1,n_2)=(3, 5)$] Let $\alpha$ and $\beta$ have degrees $(n_1,n_2)=(3, 5)$, and $m=\frac{z_3}{z_{5}}$ is the multiplier for degree $n=\frac{5}{3}$, then 
	\begin{enumerate}[(i)]
		\item natural form	
		\begin{equation}\label{eq40}
			\begin{aligned}
				m-\frac{5}{3m} =&\frac{2}{3}  \left((\alpha \beta)^{1 / 8}-\{(1-\alpha)(1-\beta)\}^{1 / 8}\right) \left[1-3\left((\alpha \beta)^{1 / 8}+\{(1-\alpha)(1-\beta)\}^{1 / 8}\right)\right.\\
				&\left.+3\left((\alpha \beta)^{1 / 4}+\{(1-\alpha)(1-\beta)\}^{1 / 4}\right)+2 (\alpha \beta (1-\alpha)(1-\beta))^{1 / 8}\left(3-(\alpha \beta)^{1 / 8}-\{(1-\alpha)(1-\beta)\}^{1 / 8}\right) \right],
			\end{aligned}
		\end{equation}		
		\item emphatic form		
		\begin{equation}\label{eq41}	
			\begin{aligned}
				m-\frac{5}{3m} =&\frac{2}{3} \left((\alpha \beta)^{1 / 4}-\{(1-\alpha)(1-\beta)\}^{1 / 4}\right) \left[4\sqrt{\frac{1}{2}(1+\sqrt{\alpha \beta}+\sqrt{(1-\alpha)(1-\beta)})}\right.\\
				&\left.-4+\left((\alpha \beta)^{1 / 4}+\{(1-\alpha)(1-\beta)\}^{1 / 4}\right) \right],
			\end{aligned}
		\end{equation}	
		\item 	Ramanujan’s theta functions form	
		\begin{equation}\label{eq42}	
			\begin{aligned}
				\frac{\varphi^2(q^3)}{\varphi^2(q^{5})}-\frac{5}{3} \frac{\varphi^2(q^{5})}{\varphi^2(q^3)}=&\frac{2}{3} \left(\frac{4 q^2 \psi(q^6) \psi (q^{10})}{\varphi (q^3) \varphi \left(q^5\right)}-\frac{\varphi (-q^3) \varphi (-q^5)}{\varphi (q^3) \varphi \left(q^5\right)}\right) \left[4\left(\frac{\varphi (q^6) \varphi (q^{10}))}{\varphi (q^3) \varphi \left(q^5\right)}+\frac{4 q^4 \psi (q^{12})) \psi (q^{20}))}{\varphi (q^3) \varphi \left(q^5\right)}\right)\right.\\
				&\left.-4+\left(\frac{4 q^2 \psi(q^6) \psi (q^{10})}{\varphi (q^3) \varphi \left(q^5\right)}+\frac{\varphi (-q^3) \varphi (-q^5)}{\varphi (q^3) \varphi \left(q^5\right)}\right) \right].
			\end{aligned}
		\end{equation}	
	\end{enumerate}
	
\end{theorem}
The proof of Theorem 2.2 is similar to the proof of the Theorem 2.1, hence we omit the details.

\section*{Acknowledgment}
The authors would like to thank the anonymous referee for many invaluable suggestions. 





\bibliographystyle{elsarticle-num}
\bibliography{Ramanujan_of_15}



	
	
	

\end{document}